\documentclass[10pt,a4paper,english]{amsart}
\usepackage[utf8]{inputenc}

\usepackage{amsmath,amsfonts,amsthm,amssymb}
\usepackage{url}
\usepackage{esint}
\usepackage{fullpage}
\usepackage[backend=biber]{biblatex}
\usepackage{tikz}
\usepackage{enumitem}

\DeclareFontFamily{U}{mathx}{\hyphenchar\font45}
\DeclareFontShape{U}{mathx}{m}{n}{
      <5> <6> <7> <8> <9> <10>
      <10.95> <12> <14.4> <17.28> <20.74> <24.88>
      mathx10
      }{}
\DeclareSymbolFont{mathx}{U}{mathx}{m}{n}
\DeclareFontSubstitution{U}{mathx}{m}{n}
\DeclareMathAccent{\widecheck}{0}{mathx}{"71}
\DeclareMathAccent{\wideparen}{0}{mathx}{"75}

\newtheorem{theorem}{Theorem}[section]
\newtheorem{proposition}[theorem]{Proposition}

\newtheorem{lemma}[theorem]{Lemma}

\newtheorem{assumption}[theorem]{Assumption}

\let\Im\relax

\DeclareMathOperator*{\Im}{Im}

\newcommand{\lela}{\left \langle}
  \newcommand{\rira}{\right \rangle}

\newcommand{\R}{\mathbb R}
\newcommand{\N}{\mathbb N}
\newcommand{\Z}{\mathbb Z}
\newcommand{\C}{\mathbb C}

\newcommand{\cF}{\mathcal F}

\newcommand{\japx}{\langle x \rangle}

\begin{document}

\title{Finite-size effects in response functions of molecular systems}
\author{Mi-Song Dupuy$^{*}$, Antoine Levitt$^{\dagger}$}
\maketitle
\begin{abstract}
  We consider an electron in a localized potential submitted to a weak
  external, time-dependent field. In the linear response regime, the
  response function can be computed using Kubo's formula. In this
  paper, we consider the numerical approximation of the response
  function by means of a truncation to a finite region of space. This
  is necessarily a singular approximation because of the discreteness
  of the spectrum of the truncated Hamiltonian, and in practice a
  regularization (smoothing) has to be used. Our results provide error
  estimates for the response function past the ionization threshold
  with respect to both the smoothing parameter and the size of the
  computational domain.
\end{abstract}
\tableofcontents

\section{Introduction}
Consider a molecule in its electronic ground state, to which an
external time-dependent electric field is applied. The resulting
change in the electronic density can be computed using linear response
theory, resulting in a quantity $\widehat K(\omega)$ describing the
response at frequency $\omega$. To compute it in practice, the domain
of computation has to be truncated to a region of size $L$, yielding
an approximate response function $\widehat K_{L}(\omega)$. Since the
dynamics on the full space and on a finite region of space are
qualitatively different, $\widehat K_{L}$ is qualitatively different
from $\widehat K$: in particular, even when $\widehat K$ is a regular
function, $\widehat K_{L}$ is always a singular distribution,
reflecting the discreteness of the spectrum of the Hamiltonian. This
paper answers, in a simplistic one-electron model, the following
question: in which sense does $\widehat K_{L}$ converge to
$\widehat K$, and with what convergence rate?

We focus here for technical convenience on the simplest
continuous model of a single-electron system in a potential $V$,
described by the rest Hamiltonian
\begin{align}
  H = - \Delta + V,
\end{align}
with $V$ decaying at infinity in a sense to be made precise. The
Hamiltonian $H$ is self-adjoint on $L^{2}(\R^{d})$, with possible
negative eigenvalues and continuous spectrum $[0, \infty)$. Assume
that there is a simple lowest eigenvalue $E_{0} < 0$, with associated
eigenfunction $\psi_{0}$. We consider the time-dependent Schrödinger
equation
\begin{align}
  i \partial_{t} \psi = H\psi + \varepsilon f(t) V_{\mathcal P} \psi, \quad \psi(0) = \psi_{0}
\end{align}
where $V_{\mathcal P}$ is a perturbing potential, $f$ a continuous causal function (\emph{i.e.} $f(t) = 0$ for $t<0$), and
$\varepsilon > 0$ is a small parameter. If $V_{\mathcal O}$ is a potential
representing an observable, to first order in $\varepsilon$, we have
for all $t \in \R$
\begin{align}
  \langle \psi(t), V_{\mathcal O} \psi(t) \rangle = \langle \psi_{0}, V_{\mathcal O} \psi_{0} \rangle + \varepsilon (K \ast f)(t) + O(\varepsilon^{2}).
\end{align}
The function $K$ is the response function, computed in Theorem
\ref{thm:Kubo}. For instance, when $V_{\mathcal P} = -x_{\beta}$ and
$V_{\mathcal O} = x_{\alpha}$, then $K(t)$ is the polarizability impulse response:
the dipole response at time $t$ in the direction $\alpha$ of the system to an
impulse uniform field at time $0$ in the direction $\beta$.

$K$ is a continuous causal function of at most polynomial growth, and
has a distributional Fourier transform
\begin{align}
  \label{eq:Komega}
  \widehat K(\omega) &= \lim_{\eta \to 0^{+}} \lela \psi_{0}, V_{\mathcal O} \Big(\omega +i\eta - (H-E_0)\Big)^{-1} V_{\mathcal P} \psi_{0}\rira - \lela \psi_{0},V_{\mathcal P} \Big(\omega +i\eta + (H-E_0)\Big)^{-1} V_{\mathcal O} \psi_{0}\rira,
\end{align}
where the limit is taken in the sense of distributions, and $\eta \to
0^{+}$ means the one-sided limit as $\eta$ converges to zero by positive values. This quantity
contains the frequency information of the response, a valuable
physical output. Using a spectral resolution of
$H = \int_{\R} \lambda \,\mathrm{d}P(\lambda),$ where $\mathrm{d}P(\lambda)$ is a
projection-valued measure, one can formally rewrite it as
\begin{align*}
  \widehat K(\omega) = \lim_{\eta \to 0^{+}} \int_{\R} \frac{\langle  V_{0} \psi_{0}, \mathrm{d}P(\lambda) V_{\mathcal P} \psi_{0} \rangle}{\omega + i\eta - (\lambda-E_{0})} - \frac{\langle  V_{\mathcal P} \psi_{0}, \mathrm{d}P(\lambda) V_{\mathcal O} \psi_{0} \rangle}{\omega + i\eta + (\lambda-E_{0})}.
\end{align*}
The distributional limit (sometimes called
Plemelj-Sokhotski formula)
\begin{align}
  \label{eq:plemelj}
  \lim_{\eta \to 0^{+}} \frac 1 {x+i\eta} = \lim_{\eta \to 0^{+}} \frac{x}{x^{2} + \eta^{^{2}}} - i \frac{\eta}{x^{2}+\eta^{2}} = {\rm p.v.} \frac 1 x - i \pi \delta_{0},
\end{align}
where ${\rm p.v.}$ stands for the Cauchy principal value shows
that $\widehat K$ is a singular distribution at the excitation
energies $\omega = \pm (E_{n} - E_{0})$, where $E_{n}$ are the
eigenvalues of $H$ other than $E_{0}$. Past $\lambda > 0$ however, the
spectrum of $H$ is continuous, and therefore for $|\omega| > -E_{0}$,
the nature of $\widehat K$ depends on that of
$\langle V_{\mathcal O} \psi_{0}, \mathrm{d}P(|\omega| - E_{0}) V_{\mathcal P} \psi_{0} \rangle$.
Under certain conditions, one can prove that this quantity is regular:
this is one avatar of a limiting absorption
principle. Such principles have a long history in mathematical
physics, and are a first step towards scattering theory
\cite{agmon1975spectral,reed1978iii}. Physically, this corresponds to
ionization: the electron, under the action of the forcing field,
dissolves into the continuum and goes away to infinity.

Consider now a box $[-L,L]^{d}$ with Dirichlet boundary conditions,
giving rise to a (semi-)discretized operator $H_{L}$. In practice,
this box is further discretized onto a grid for instance; however the
convergence as a function of the grid size is a different, more
standard problem, which we do not consider in this paper. From $H_{L}$
we can define a response function $K_{L}$ and its Fourier transform $\widehat K_{L}$, similarly to the
definition of $K$ and $\widehat K$ in \eqref{eq:Komega}. Note that $H_{L}$ has
compact resolvent and a discrete set of eigenvalues, tending to
infinity. Therefore $\widehat K_{L}$ is a singular distribution,
reflecting the fact that complete ionization is not possible in a
finite system. A smooth function can be obtained by computing
$\widehat K_{L}(\omega+i\eta)$ at finite $\eta$, which blurs the
discrete energy levels into a continuum, and physically corresponds to
adding an artificial dissipation. This however results in a distortion
of the true response function. In physically relevant
three-dimensional computations, for instance using time-dependent
density functional theory, obtaining converged spectra requires a
manual selection of an appropriate $\eta$ parameter. Furthermore, only
moderate values of $L$ can realistically be taken, and convergence is
often slow and unpredictable \cite{d2019locality}. The main
contribution of our paper is to clarify in which sense $\widehat K_{L}$
converges to $\widehat K$, and to quantify sources of error due to
finite $\eta$ and $L$.

The mathematical and numerical analysis of ground state properties of
molecular systems is by now relatively well established. At finite
volume the convergence of a number of numerical methods for various
mean-field models has been established \cite{cances2012numerical}.
Finite-size effects have been studied mathematically in periodic
systems \cite{gontier2017supercell,cances2020numerical}. However,
although a number of authors have focused on establishing the validity
of linear response theory
\cite{bouclet2005linear,bachmann2018adiabatic,teufel2020non,cances2020coherent},
and studying its properties
\cite{cances2012mathematical,prodan2013quantum}, work on the numerical
analysis of response quantites remains fairly scarce. In
particular, we believe our work to be the first to address rigorously
the important question of ionization in this context.

\section{Notations and assumptions}
We work in $d$ space dimensions. Following conventions usual in
quantum mechanics, we use
\begin{align*}
  \widehat f(\omega) = \int_{\R} e^{i\omega t} f(t) dt, \quad (\mathcal F f)(q) = \int_{\R^{d}} e^{-i q \cdot x} f(x) dx
\end{align*}
for the Fourier
transforms in time and space respectively. The unusual sign in the time Fourier
transform is done so that the elementary solution $e^{-i E t}$ to
the Schrödinger equation has a Fourier transform localized on $\{E\}$.

For $k \in \mathbb{N}, 0 \leq \alpha \leq 1$, we will note
$C^{k,\alpha}$ the space of $k$ times differentiable functions with a
Hölder $\alpha$ continuous $k$-th derivative. We denote by
$L^{2}(\R^{d})$ the Lebesgue space, by $H^{k}(\R^{d})$ the Sobolev
space, by $\mathcal S(\R^{d})$ the space of Schwartz functions and by
$\mathcal S'(\R^{d})$ the space of tempered distributions. When left
unspecified, $\|\cdot\|$ refers to the $L^{2}(\R^{d})$ norm. For a
weight function $w : \R^{d} \to (0,\infty)$, we denote by
\begin{align*}
  L^{2}(w) &= \left\{\psi, \,\psi w \in L^{2}(\R^{d})\right\}\\
  H^{k}(w) &= \left\{\psi, \,\psi w \in H^{k}(\R^{d})\right\}
\end{align*}
the weighted spaces, with naturally associated Hilbert space
structure. We use the Japanese bracket convention
$\japx = \sqrt{1 + |x|^{2}}$ for the regularized norm. Spaces of particular interest are
$L^{2}(\japx^{n})$, the space of polynomially decaying functions of
exponent $n$, and $L^{2}(e^{\alpha \japx})$ and exponentially decaying
functions with rate $\alpha$. We will use in proofs only the notation
$a \lesssim b$ to mean that there exists $C > 0$ such that
$a \le C b$, where the dependence of $C$ on other quantities is made
clear in the statement to be proved.

\medskip

We first assume a strong regularity on $V$.
\begin{assumption}[Smoothness of $V$]
  \label{assump:V_differentiability}
  The potential $V: \R^{d} \to \R$ is smooth with bounded derivatives.
\end{assumption}
This strong assumption is only required to establish the existence of
a propagator in Theorem \ref{thm:Kubo}, using the results of
\cite{fujiwara1979construction}, and of the linear response function $\widehat{K}$ given by the Kubo formula (Equation~\eqref{eq:K_frequency}). 
It could significantly be relaxed for the other results in this paper, as our focus is on the properties of $\widehat{K}$, which are related to the behavior at infinity of the potential.  


\medskip

More important are the decay properties of $V$.
\begin{assumption}[Decay of $V$]
    \label{assump:decay_V}
  There is $\varepsilon > 0$ such that $|x|^{2+\varepsilon} V(x)$ is bounded.
\end{assumption}
This assumption is to establish the differentiability of the resolvent
on the boundary; see remarks after our main result for possible
extensions to potentials decaying less quickly.

Under these two assumptions, as is standard, $H$ has domain
$H^{2}(\R^{d})$, and continuous spectrum $[0, \infty)$; in particular,
there are no embedded eigenvalues in $[0,\infty)$ \cite[Theorem
XIII.58]{reed1978iv}.

\medskip

\begin{assumption}[Non-degenerate ground state]
  There is at least one negative eigenvalue. The lowest eigenvalue
  $E_{0}$ is simple. We denote by $\psi_{0}$ the unique (up to sign)
  associated eigenfunction.
\end{assumption}
We establish our results for the ground state for concreteness, but
this is not crucial: the same results would be valid for any simple
eigenvalue.

\medskip

\begin{assumption}[Observable and perturbation]
  \label{assump:observables_and_perturb}
  The observable $V_{\mathcal O} : \R^{d} \to \R$ and perturbation $V_{\mathcal P} : \R^{d}
  \to \R$ are infinitely
  differentiable and sub-linear: for all $|\alpha| \ge 1$,
  $\partial^{\alpha} V_{\mathcal O}$ and $\partial^{\alpha} V_{\mathcal P}$ are bounded.
\end{assumption}
In particular this allows the potentials $x_{i}$, in which case the
response functions are the dynamical polarizabilities. Again this is
to establish the existence of a propagator in Theorem \ref{thm:Kubo}.
Our results from then on only require potentials growing at most
polynomially, and could also be extended to accomodate more general
operators (such as the current operator).

\section{Main results}
\subsection{Kubo's formula}
We first give Kubo's formula in our context and define the response
function $K$.
\begin{theorem}[Kubo]
  \label{thm:Kubo}
  For all continuous and causal functions
  $f : \R \to \R$ uniformly bounded by $1$, for all
  $0 < \varepsilon < 1$, the Schrödinger equation
  \begin{align*}
    i\partial_{t} \psi = H \psi + \varepsilon f(t) V_{\mathcal P} \psi , \quad \psi(0) = \psi_{0}
  \end{align*}
  has a unique strong solution for all times. Furthermore,
  \begin{align}
    \label{eq:kubo_formula}
    \langle  \psi(t), V_{\mathcal O} \psi(t) \rangle = \lela \psi_{0}, V_{\mathcal O}  \psi_{0}\rira + \varepsilon (K \ast f)(t) + R_{\varepsilon}(t)
  \end{align}
  with
  \begin{align*}
    |R_{\varepsilon}(t)| \le C \varepsilon^{2} (1+|t|^{8})
  \end{align*}
  for some $C > 0$ independent of $t, \varepsilon$. The response
  function $K$ is defined by
  \begin{align}
    \label{eq:kubo_time}
    K(\tau) &= -i \theta(\tau) \lela V_{\mathcal O} \psi_{0},  e^{-i (H-E_{0}) \tau} V_{\mathcal P} \psi_{0}\rira + {\rm c.c.},
  \end{align}
  where $z + {\rm c.c.}$ is a notation for $z + \overline{z}$, and
  $\theta$ is the Heaviside function.
  It is continuous, of at most polynomial growth, and causal.
\end{theorem}
The proof of this theorem is given in Section \ref{sec:Kubo}. The
expression for $K$ results from a Dyson expansion, and the bound on
$R_{\varepsilon}(t)$ from a control of the growth of moments of
$\psi(t)$ using the commutator method.

\medskip

Since $K$ is causal and of at most polynomial growth, one can define
its Fourier transform in two different senses: as a tempered
distribution $\widehat K(\omega)$ on the real line (defined by duality
against Schwartz functions), and as a holomorphic function
$\widehat K(z)$ on the open upper-half complex plane (defined by the
convergent integral
$\int_{0}^{+\infty} K(\tau) e^{iz} \mathrm{d} \tau$). Since
$K(\tau) e^{-\eta \tau}$ converges towards $K$ in the sense of
tempered distributions, both these definitions agree in the sense that
\begin{align*}
  \widehat K(\omega) = \lim_{\eta \to 0^{+}} \widehat K(\omega+i\eta)
\end{align*}
in the sense of tempered distributions. 

Using for $\eta > 0$
\begin{align*}
  \int_{0}^{+\infty} e^{i (\omega+i\eta-\lambda) \tau} \mathrm{d}\tau = \frac{i}{\omega+i\eta-\lambda}
\end{align*}
and functional calculus, it follows that
\begin{align}
  \label{eq:K_frequency}
  \widehat K(\omega) &= \lim_{\eta \to 0^{+}} \lela V_{\mathcal O} \psi_{0}, \Big(\omega +i\eta - (H-E_0)\Big)^{-1} V_{\mathcal P} \psi_{0}\rira - \lela V_{\mathcal P} \psi_{0}, \Big(\omega +i\eta + (H-E_0)\Big)^{-1} V_{\mathcal O} \psi_{0}\rira
\end{align}
in the sense of tempered distributions.

\subsection{The limiting absorption principle}
When $|\omega| \notin \sigma(H) - E_{0}$, $\widehat K(\omega)$ defines
an analytic function in a neighborhood of $\omega$. When $\omega = E_{n} - E_{0}$
for $E_{n}$ an eigenvalue of $H$, $\lim_{\eta \to 0^{+}} \widehat
K(\omega)$ diverges, and the distribution $\widehat K$ is singular at
$\omega$. When $|\omega| > -E_{0}$, \emph{i.e.} above the ionization
threshold, we have the following result.
\begin{theorem}
  \label{thm:main1}
  The tempered distribution $\widehat K$ is a continuously differentiable
  function for $|\omega| > -E_{0}$. Furthermore, for all such $\omega$
  there is $C > 0$ such that for all $0 < \eta < 1$,
  \begin{align}
    \label{eq:main_1}
    |\widehat K(\omega+i\eta) - \widehat K(\omega)| &\le C \eta.
  \end{align}
\end{theorem}
The proof of Theorem \ref{thm:main1}, in Section \ref{sec:LAP},
involves the study of the boundary values of the resolvent
$(z-H)^{-1}$ as $z$ approaches the real axis in the upper half complex
plane. This resolvent diverges as an operator on $L^{2}(\R^{d})$ as
$z$ approaches the spectrum of $H$. When $z$ approaches an eigenvalue
of $H$, this is a real divergence and the resolvent can not be defined
in any meaningful sense. However, when $z$ approaches the continuous
spectrum from the upper half plane, the divergence merely indicates a
loss of locality in the associated Green's function and, under
appropriate decay assumptions on $V$, the limit exists as an operator
on weighted spaces. This fact is known as a limiting absorption
principle, with a long history in mathematical physics; the proof we
use follows that of \cite{agmon1975spectral}.

\subsection{Discretization}
We now discretize our problem on a domain $[-L,L]^{d}$ with Dirichlet
boundary conditions. The corresponding approximations
$H_{L}, \psi_{0,L}$ and $E_{0,L}$ give rise to an approximate response
function $K_{L}$ (see exact definitions in Section
\ref{sec:truncation}). Our main result is then:
\begin{theorem}
  \label{thm:main_2}
  $K_{L}$ converges towards $K$ in the sense of tempered distributions.
  Furthermore, for all $\omega \in \R$ there are $\alpha > 0$, $C > 0$
  such that for all $0 < \eta < 1$, $L > 0$,
  \begin{align}
    \label{eq:main_2}
    |\widehat K_{L}(\omega+i\eta) - \widehat K(\omega+i\eta)| &\le C \frac{e^{-\alpha \eta L}}{\eta^{2}}
  \end{align}
\end{theorem}
The proof of this theorem is given in Section \ref{sec:truncation}.
When $|\omega| < -E_{0}$ is not equal to a difference of eigenvalues,
the bound $e^{-\alpha \eta L}$ is pessimistic, and the decay rate is
actually independent of $\eta$ (as can be seen from the proof).

The convergence of $K_{L}$ towards $K$ in the sense of
distributions (i.e. when integrated against a quickly decaying
function of time) can be heuristically understood in as follows: since
the initial condition $\psi_{0}$ is localized close to the origin, for
moderate times (compared to some power of $L$) finite size effects are
not relevant; only for longer times (damped by the test function) will
the reflections against the boundary affect the value of $K_{L}$. To
obtain \eqref{eq:main_2}, we note that at a fixed $\eta > 0$, the
resolvent $(\lambda + i\eta - H)^{-1}$ is a well-defined operator, and
its kernel $G(x,y)$ decays exponentially for large $||x|-|y||$, with a
decay rate proportional to $\eta$. Since $\psi_{0}$ is exponentially
localized, the quantity $\widehat K(\omega+i\eta)$ only involves
quantities localized on a region of space of size of order $1/\eta$,
and can therefore be computed accurately when $L \gg 1/\eta$, leading
to our result.

It follows from the two results above that one can approximate
$\widehat K(\omega)$ for $|\omega| > -E_{0}$ by taking the limit
$L \to \infty$ (at finite $\eta$) then $\eta \to 0$, but not the
reverse. At a fixed box size $L$, the optimal $\eta$ is the one that minimizes
the total error $\frac{e^{-\alpha \eta L}}{\eta^{2}} + \eta$: up to
logarithmic factors, it is of order $\tfrac 1 L$, and so is the total error.

\subsection{Remarks}

\subsubsection{Decay of the potential and regularity of $\widehat K$.} Our
assumption that $|x|^{2+\varepsilon} V(x)$ is bounded guarantees that
$|\widehat K(\omega+i\eta) - \widehat K(\omega)|$ is of order $\eta$.
We actually show in our proof the stronger result that, if
$|x|^{1+k+\alpha+\varepsilon} V(x)$ is bounded for some
$k \in \N, \alpha \in [0,1]$, then
\begin{align*}
  \widehat K \in C^{k,\alpha}\Big(((-\infty, E_{0}) \cup (-E_{0}, +\infty)) + i [0,+\infty]\Big).
\end{align*}
However, long-range potentials (decaying like $1/|x|$) are not covered
by the results in this paper, due to the absence of a limiting
absorption principle in this case. Indeed, even showing the absence of
embedded eigenvalues becomes a delicate matter \cite{reed1978iv}. To
our knowledge, a limiting absorption principle with long-range
potentials has been proved only in the radial case
\cite{agmon-klein1992LAP_radial}.

\subsubsection{Higher order approximations.} In the common case
where $V_{\mathcal O} = V_{\mathcal P}$, it follows from the
Plemelj-Sokhotski formula
\eqref{eq:plemelj} that the imaginary part of
$\widehat K(\cdot+i\eta)$ is the convolution of the
imaginary part of $\widehat K$ with a Lorentzian profile of width
$\eta$ and height $1/\eta$, an approximation of the Dirac distribution. In
general, if $\phi$ is a Schwartz function of integral $1$,
$\phi_{\eta}(x) = \phi(x/\eta)/\eta$ and if $f$ is of class $C^{p+1}$
near $\omega$, then
  \begin{align*}
    (f \ast \phi_{\eta})(\omega) = f(\omega) + O(\eta^{p+1}),
  \end{align*}
  where the order $p$ of $\phi$ is the smallest integer such that
  $\int x^{p'} \phi(x) dx = 0$ for all $0 < p' \le p$ (see for instance
  \cite[Section 5.1]{cances2020numerical}). Since the Lorentzian
  kernel is even, we would naively expect an error proportional to
  $\eta^{2}$; however, the Lorentzian
  kernel has heavy tails (decaying like $1/x^{2}$) and therefore the
  error is only of order $\eta$ in general.

  When $V$ decays sufficiently rapidly, the above analysis suggests
  the possibility of using different kernels, such as a Gaussian
  kernel, or even a higher-order one. Such a possibility has to the
  best of our knowledge not been explored in the literature.
  
\subsubsection{Extensions.}
  We have here considered the one-electron model with a given
  Hamiltonian $H = -\Delta + V$ acting on $L^{2}(\R^{d})$. The following extensions can be
  considered
  \begin{itemize}
  \item We could consider models of the type $H = H_{0} + V$ with more
    general $H_{0}$. For instance, one can think of periodic operators
    $H_{0} = -\Delta + V_{\rm per}$, or lattice models acting on
    $\ell^{2}(\Z^{d})$, both of which can be analyzed using the Bloch
    transform. Extending our results needs two ingredients. The first
    is the error analysis of the effect of truncation on resolvents
    and on eigenvectors, which is complicated by the possibility of
    spectral pollution (see \cite{cances2014non}). The second is a
    limiting absorption principle for $H_{0}$. Following the proof in
    Section~\ref{sec:LAP}, this can be done at regular values
    of the dispersion relation, so that the energy isosurfaces form a
    smooth manifold over which a trace theorem can be established; see
    \cite{radosz2010principles} and references therein.
  \item We could consider models of several electrons. Our results can
    straightforwardly be extended to the case of non-interacting
    electrons, in which case the response function is simply a sum of
    one-electron response functions. The case of interacting electrons
    (using either the full many-body model, or mean-field models such as
    the Hartree model or time-dependent density functional theory with
    adiabatic exchange-correlation potentials) requires more care, and
    would be an interesting topic for further research.
  \end{itemize}
  
\subsubsection{Boundary conditions.}
  We here use Dirichlet boundary conditions; this is done for
  conceptual simplicity, and because Dirichlet boundary conditions
  yield a conforming scheme (in the sense that the eigenfunctions
  obtained at finite $L$ are valid trial functions for the whole-space
  problem). Using Neumann or periodic boundary conditions would
  presumably yield a similar result, but the mathematical analysis is
  slightly more involved.

  More interesting is the use of ``active'', frequency-dependent
  boundary conditions, designed to better reproduce the continuous
  spectrum. Such boundary conditions, conceptually based on an exact
  solution of the free resolvent outside a computational domain, are
  widely used in scattering problems (absorbing
  boundary conditions, perfectly matched layers \cite{antoine2017friendly}) and in the study of
  resonances in quantum chemistry (complex scaling, complex absorbing
  potential \cite{stefanov2005approximating, muga2004complex}).
\section{Numerical illustration}
We illustrate our results with a simple model. Instead of a
continuous model, we choose a discrete tight-binding model, set on
$\ell^{2}(\Z)$, with Hamiltonian
\begin{align*}
  H_{mn} = \delta_{m,n+1} + \delta_{m,n-1} + V \delta_{m,n} \delta_{n,0}.
\end{align*}
The first two terms (``hopping terms'') are analogous to a
kinetic energy and describe the motion of a particle to neighboring
sites. The third term, a compact perturbation of the free Hamiltonian,
is an impurity potential on site $0$.

This operator has continuous spectrum $[-2,2]$. We choose
$V=-4$, which leads to a single negative eigenvalue $E_{0} \approx -4.47$. We choose both
for the perturbing potential $V_{\mathcal P}$ and for the observable $V_{\mathcal O}$ the potential
$\delta_{n0}$ localized on site $0$.

To compute $K_{L}$, we truncate the Hamiltonian to a finite set of
$2L+1$ sites $\{-L,\dots,L\}$, with Dirichlet boundary conditions and
diagonalize the resulting Hamiltonian $H_{L}$ to obtain the eigenpairs
$(\psi_{n,L},E_{n,L})$ for $n=0,\dots,2L$, ordered by increasing
eigenvalue. The expression for $K_{L}$ and $\widehat K_{L}$ can be
expanded in this basis, turning into ``sum-over-states'' formulas.

\begin{figure}[h!]
  \centering
  \includegraphics[width=.49\textwidth]{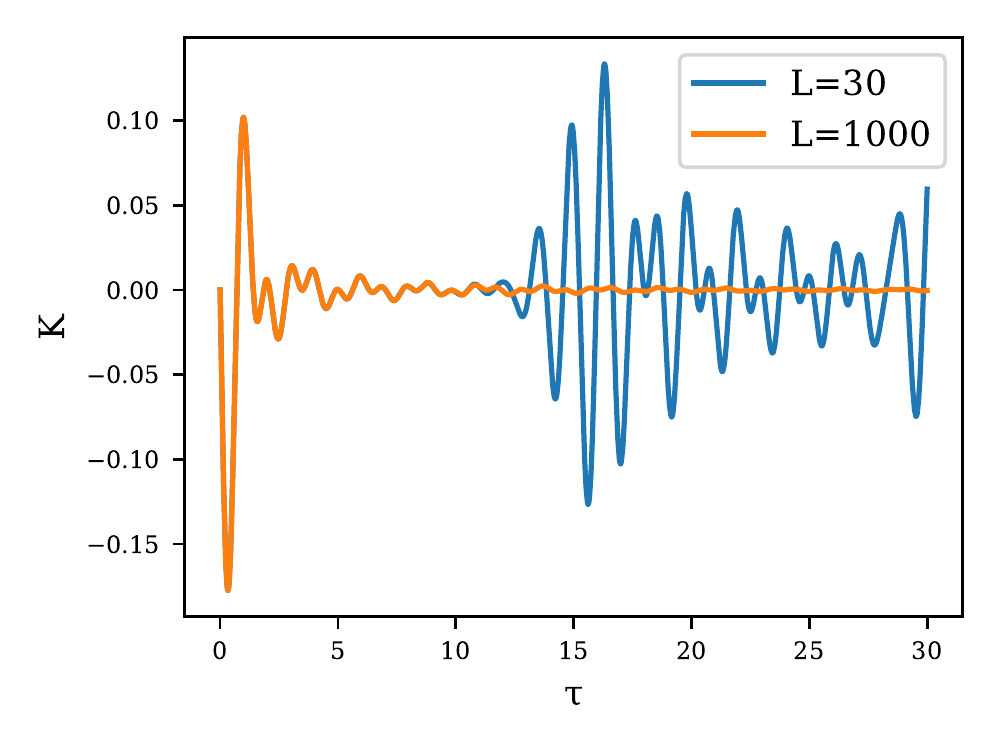}
  \caption{Time response function $K_{L}(\tau).$}
  \label{fig:fig1}
\end{figure}

We plot in Figure \ref{fig:fig1} the response function $K_{L}(\tau)$
for different values of $L$. Since $V_{\mathcal O} = V_{\mathcal P}$ and the spectrum of
$H$ is continuous except for the single bound state, the exact
response function $K(\tau)$ decays to zero, as the initial disturbance
propagates to infinity. However when observed on a finite-sized box
for long times, spurious reflections at the boundary introduce
non-decaying oscillations.
  
\begin{figure}[h!]
  \centering
  \includegraphics[width=.49\textwidth]{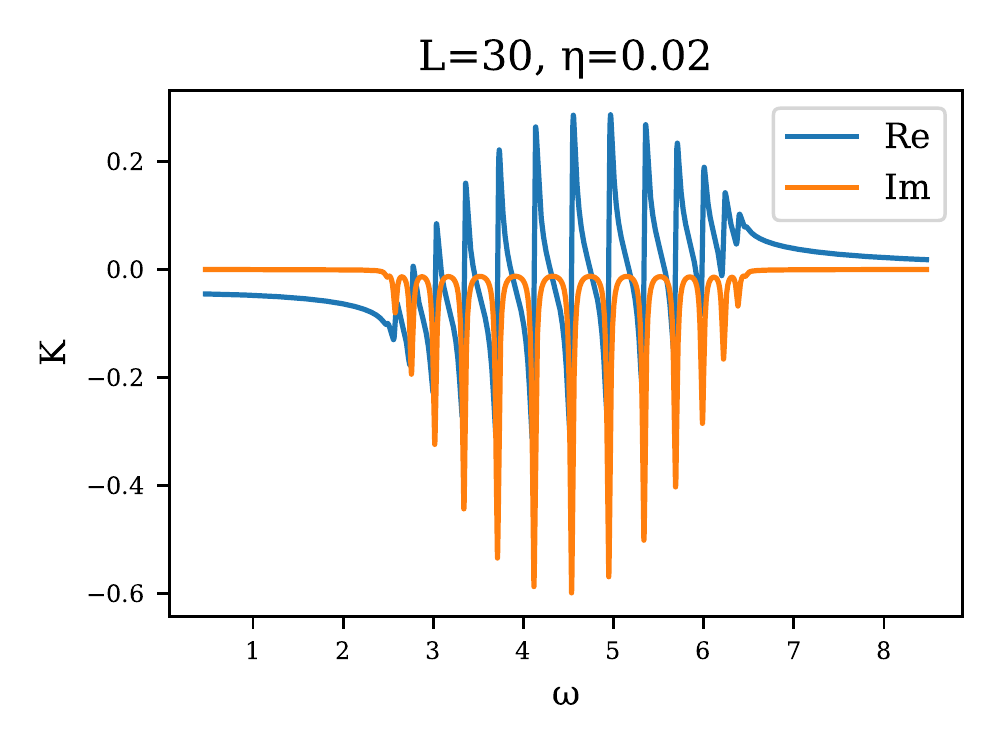}
  \includegraphics[width=.49\textwidth]{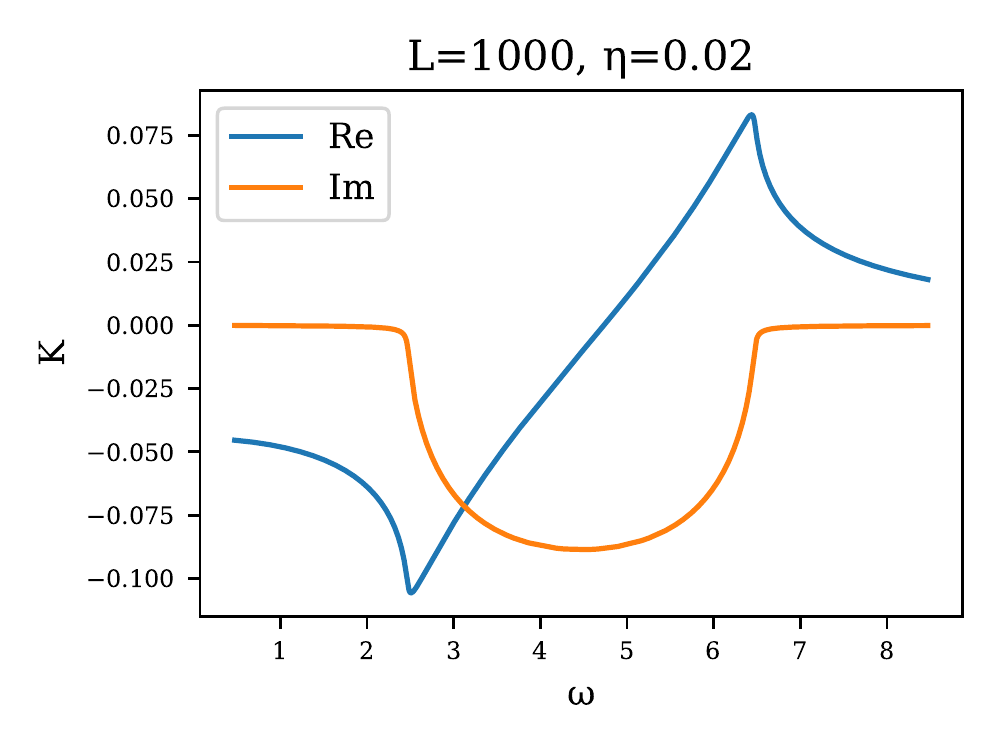}
  
  \includegraphics[width=.49\textwidth]{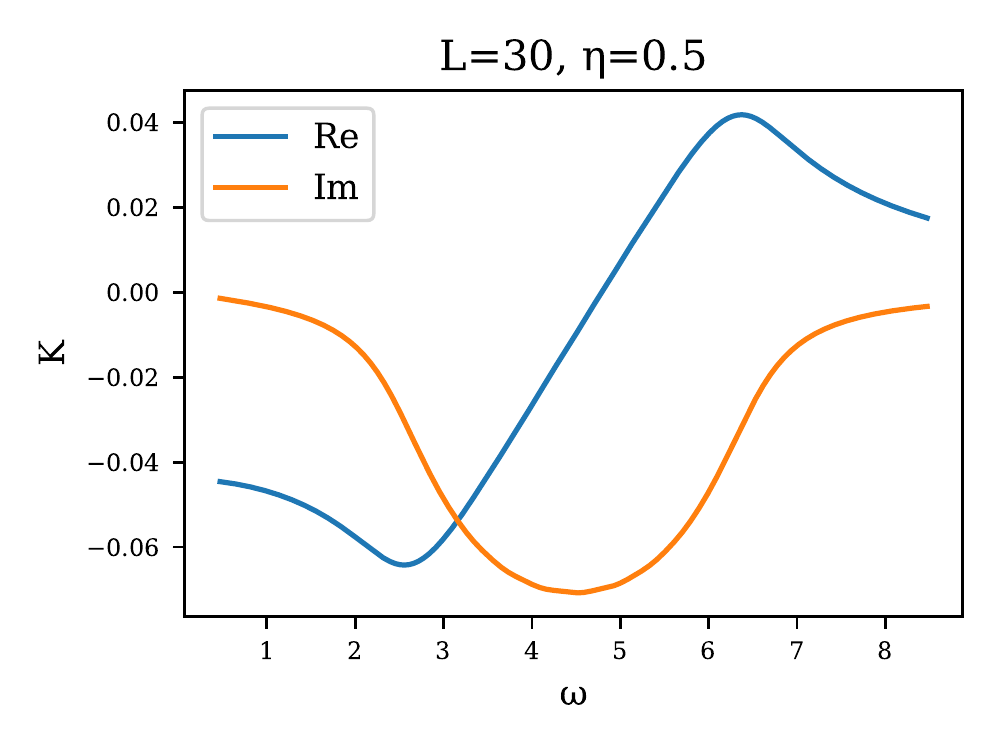}
  \includegraphics[width=.49\textwidth]{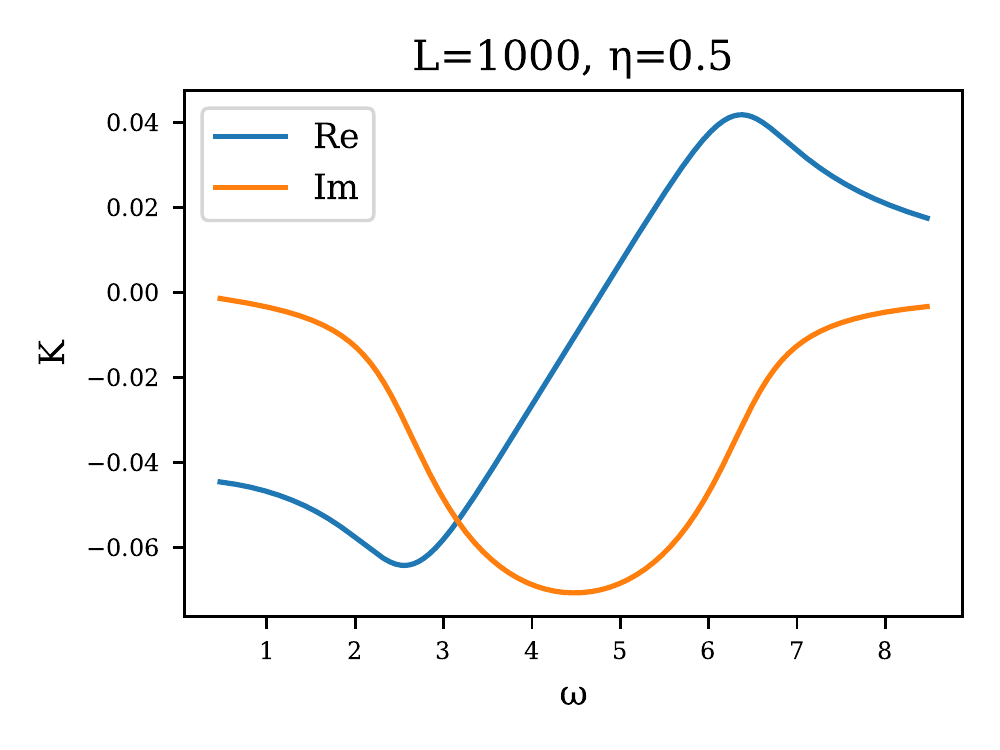}

\caption{Frequency response function $\widehat{K}_{L}(\omega+i\eta)$.}
  \label{fig:fig2}
\end{figure}

This same phenomenon can be seen in frequency space in Figure
\ref{fig:fig2}, where we plot the frequency response function
$\widehat K_{L}(\omega+i\eta)$ for different values of $\eta$ and $L$.
We plot the region $\omega \in [0, 9]$, which contains the region
$[-2, 2] - E_{0}$ corresponding to ionization; not represented is the
other ionization region $E_{0} - [-2, 2]$. When $\eta$ is small and
$L \ll 1/\eta$, the discrete nature of the spectrum is evident, and
the response function is composed of individual peaks. When
$L \gg 1/\eta$, these peaks are blurred into a continuous function.
Higher $\eta$ result in more accurate functions at moderate $L$, at
the price of over-smoothing.
\section{The Kubo formula}
\label{sec:Kubo}
We begin by studying the eigenfunction $\psi_{0}$ associated to the
eigenvalue $E_{0}$.
  \begin{lemma}
    \label{lem:psi0_exp_loc}
    There is $\alpha_{0} > 0$ such that $\psi_{0} \in H^{2}(e^{\alpha_{0} \japx})$.
  \end{lemma}
  \begin{proof}
        Since $V$ decays at infinity, for all $\varepsilon > 0$, we can write
$V = V_{c} + V_{\varepsilon}$ with $V_{c}$ compactly supported and
$\|V_{\varepsilon}\|_{L^{\infty}(\R^{d})} \le \varepsilon$. Then, for
$\varepsilon\le -E_{0}/2$ we can write
  \begin{align*}
    \psi_{0} = - (-\Delta + V_{\varepsilon} - E_{0})^{-1} V_{c} \psi_{0}.
  \end{align*}
  Since $V_{c}$ is compactly supported, $V_{c} \psi_{0}$ is in
  $L^{2}(e^{\alpha \japx})$ for all $\alpha > 0$, and so by Lemma
  \ref{lem:resolvent}, $\psi_{0}$ belongs to
  $H^{2}(e^{\alpha_{0} \japx})$ for some $\alpha_{0} > 0$ small
  enough.
\end{proof}
Note that this estimate is not sharp since the actual decay rate of
$\psi_{0}$ is $\sqrt{-E_{0}}$ (which can be obtained by sharper
Combes-Thomas estimates), but this will be sufficient for our
purposes.

We now prove Kubo's formula.
\begin{proof}[Proof of Theorem \ref{thm:Kubo}]
  Let $U_0(t,s) = e^{-i H (t-s)}$ be the unitary propagator of the
  unperturbed Hamiltonian $H=-\Delta+V$, and $U_\varepsilon(t,s)$ that of
  the perturbed Hamiltonian $H_\varepsilon(t)=-\Delta+V+\varepsilon f(t) V_{\mathcal P}$,
  whose existence is guaranteed by Lemma~\ref{lem:growth-dynamics}. By
  the Duhamel/variation of constant formula,
  \begin{align*}
    \psi(t)&= \underbrace{e^{-i E_{0} t}\psi_{0}}_{\psi^{0,0}(t)} \;\underbrace{- i \varepsilon \int_{0}^{t} U_{0}(t,t') f(t') V_{\mathcal P} U_{\varepsilon}(t',0) \psi_{0}dt'}_{\varepsilon \psi^{1,\varepsilon}(t)}.
  \end{align*}
  Iterating this formula, we obtain the first-order Dyson expansion
  \begin{align*}
    \psi^{1,\varepsilon}(t) = \underbrace{- i \int_{0}^{t} f(t')  U_{0}(t,t') V_{\mathcal P} U_{0}(t',0) \psi_{0}dt'}_{\psi^{1,0}(t)} \;\underbrace{-\varepsilon\int_{0}^{t}\int_{0}^{t'}  U_{0}(t,t') f(t') V_{\mathcal P} U_{0}(t',t'') f(t'') V_{\mathcal P} U_{\varepsilon}(t'',0) \psi_{0}dt' dt''}_{\varepsilon \psi^{2,\varepsilon}(t)}.
  \end{align*}
  From $\psi(t) = \psi^{0,0}(t) + \varepsilon \psi^{1,0}(t) +
  \varepsilon^{2} \psi^{2,\varepsilon}(t)$ it follows that
  \begin{align*}
    \langle \psi(t), V_{\mathcal O} \psi(t) \rangle &= \langle \psi_{0}, V_{\mathcal O} \psi_{0} \rangle +  \varepsilon \Big( \langle \psi^{1,0}(t) , V_{\mathcal O} \psi^{0,0}(t)\rangle +  \langle \psi^{0,0}(t) , V_{\mathcal O} \psi^{1,0}(t)\rangle\Big) \\
    &+ \underbrace{\varepsilon^{2} \Big( \langle \psi^{2,\varepsilon}(t) , V_{\mathcal O} \psi^{0,0}(t) \rangle +  \langle \psi^{0,0}(t) , V_{\mathcal O} \psi^{2,\varepsilon} \rangle + 2 \langle \psi^{1,0}(t), V_{\mathcal O} \psi^{1,0}(t)\rangle\Big)}_{R_{\varepsilon}(t)}
  \end{align*}
  The first-order term can be computed as
  \begin{align*}
    \langle \psi^{0,0}(t), V_{\mathcal O} \psi^{1,0}(t) \rangle + {\rm c.c.} &= \left\langle  e^{iE_{0} t} V_{\mathcal O} \psi_{0}, -i \int_{0}^{t} f(t') e^{-iH(t-t')} V_{\mathcal P} e^{-i E_{0} t'} \psi_{0} dt'\right\rangle + {\rm c.c.}\\
    &= -i\int_{0}^{t} f(t') \langle V_{\mathcal O}\psi_{0}, e^{-i(H-E_{0})(t-t')} V_{\mathcal P} \psi_{0} \rangle dt' + {\rm c.c.}.\\
    &= (K \ast f)(t).
  \end{align*}
  Since $|V_{\mathcal O}(x)| \lesssim 1+|x|$ and $\psi_{0} \in L^{2}(e^{\alpha_{0} \japx})$,
  \begin{align*}
    |R_{\varepsilon}(t)| \lesssim \|\psi^{2,\varepsilon}(t)\| + \|(1+|x|) \psi^{1,0}(t)\|.
  \end{align*}
  Using $|V_{\mathcal P}(x)| \lesssim 1+|x|$ and Lemma \ref{lem:growth-dynamics},
  we get
\begin{align*}
  \|\psi^{2,\varepsilon}(t)\| &\lesssim (1+|t|^{2})   \sup_{t' \in [0,t], t'' \in [0,t]} \|U_{0}(t,t') V_{\mathcal P} U_{0}(t',t'') V_{\mathcal P} U_{\varepsilon}(t'',0) \psi_{0}\| \\
  &\lesssim (1+|t|^{2}) \sup_{t' \in [0,t], t'' \in [0,t]} \|(1+|x|) U_{0}(t',t'') V_{\mathcal P} U_{\varepsilon}(t'',0) \psi_{0}\|\\
  &\lesssim (1+|t|^{4}) \sup_{t'' \in [0,t]} \Big( \|(1+|x|) V_{\mathcal P} U_{\varepsilon}(t'',0) \psi_{0}\| + \| \nabla (V_{\mathcal P} U_{\varepsilon}(t'',0) \psi_{0})\| \Big)\\
  &\lesssim (1+|t|^{4}) \sup_{t'' \in [0,t]} \Big( \||x|^{2} U_{\varepsilon}(t'',0) \psi_{0}\| + \| \nabla U_{\varepsilon}(t'',0) \psi_{0}\| + \|\psi_{0}\|\Big)\\
  &\lesssim (1+|t|^{8}) \left(\||x|^{2}  \psi_{0}\| + \|\Delta \psi_{0}\| + \|x \otimes \nabla \psi_{0}\| + \|\psi_{0}\|\right)
\end{align*}
The bound on $R_{\varepsilon}(t)$ then follows by establishing a bound
on $\|(1+|x|) \psi^{1,0}(t)\|$ by the same method.
\end{proof}
\section{Properties of the response function $K$}
\label{sec:LAP}

Theorem~\ref{thm:main1} is a consequence of a limiting absorption
principle for the Hamiltonian $H = -\Delta + V$ stated in
Proposition~\ref{prop:LAP_pop}. Our proof is a simplification of the
one in Agmon \cite{agmon1975spectral}, with a careful tracking of the
regularity with respect to the spectral parameter.

We begin by studying the free Laplacian.
\begin{proposition}[Limiting absorption principle for the free Laplacian]
\label{prop:LAP_free}
Let $s = \tfrac 1 2 + k + \alpha$ for $k \in \N$ and
$\alpha \in [0,1]$. The resolvent $(z+\Delta)^{-1}$ defined for
${\rm Im} z > 0$ extends to an operator of class $C^{k,\alpha}$ on the
semi-open set $(0,+\infty) + i [0,+\infty)$, in the topology of
bounded operators from $L^{2}\left(\japx^{s}\right)$ to
$H^{2}\left(\japx^{-s}\right)$.
\end{proposition}
\begin{proof}
  Let $\lambda_{0} > 0.$ Let $\chi$ be a smooth cutoff function, equal to
  $1$ in $[\lambda_{0}/2, 2\lambda_{0}]$ and to zero outside of
  $[\lambda_{0}/3,3\lambda_{0}]$. Let $\psi \in L^{2}(\japx^{s})$, and $\phi$ belong to the $L^{2}$-dual of
  $H^{2}\left(\japx^{-s}\right)$.

  Let $M_{\chi}$ be the multiplication operator in Fourier space
  defined by $\cF (M_{\chi} \psi)(q) = \chi(|q|^{2}) \cF(q)$. Then by
  spectral calculus, $(z+\Delta)^{-1} (1-M_{\chi})$ extends to a
  $C^{k,\alpha}$ operator on a set
  $[\lambda_{0}-\varepsilon, \lambda_{0}+\varepsilon] + i [0,
  +\infty)$ for $\varepsilon$ small enough, in the topology of
  operators $L^{2}(\R^{d})$ to $H^{2}(\R^{d})$. Therefore, it is
  enough to consider the term
    \begin{align}
      \notag\langle \phi, (z+\Delta)^{-1} M_{\chi} \psi\rangle &= \frac 1 {(2\pi)^{d}} \int_{\R^{d}} \frac{\chi\left(|q|^{2}\right) \cF \phi(q)^{*} \cF \psi(q)}{z - |q|^{2}} \mathrm{d}q\\
      &=\int_{\R} \frac{D_{\phi\psi}(\lambda)}{z - \lambda} \mathrm{d}\lambda,\label{eq:spectral_repr}
    \end{align}
    with the projected density of states
    \begin{align*}
      D_{\phi\psi}(\lambda) = \frac 1 {(2\pi)^{d}} \lambda^{(d-2)/2}\chi\left(\lambda\right) \int_{S^{d-1}} \cF \phi(\sqrt \lambda \hat q)^{*} \cF \psi(\sqrt \lambda \hat q) \mathrm{d}\hat q.
    \end{align*}
    Since $\mathcal{F}(M_{\sqrt \chi} \phi)$ and $\mathcal{F}(M_{\sqrt
      \chi} \psi)$ are in $H^{s}(\R^{d})$, by Lemma~\ref{lem:trace_hardy}
    $D_{\phi\psi}$ is in $H^{s}(\R)$.

    We can compute by contour integration the inverse Fourier
    transform of the function $\frac{1}{z - \cdot}$ for ${\rm Im} z > 0$:
    \begin{align*}
      \frac 1 {2\pi} \int \frac{1}{z-\lambda} e^{-i\lambda \tau} \mathrm{d}\lambda = i \theta(-\tau) e^{-iz \tau}
    \end{align*}
    Therefore, by the Parseval formula,
    \begin{align*}
      \langle \phi, (z+\Delta)^{-1} M_{\chi} \psi\rangle = 2\pi i\int_{\R^{+}} e^{iz\tau} \widecheck{D}_{\phi\psi}(\tau) \mathrm{d}\tau.
    \end{align*}
    Letting $g_{\tau}(z) = e^{i z \tau}$, it follows from
    \begin{align*}
      |g_{\tau}^{(k)}(z_{1}) - g_{\tau}^{(k)}(z_{2})| \lesssim |z_{1}-z_{2}|^{\alpha} |\tau|^{k+\alpha}
    \end{align*}
    and the Cauchy-Schwarz inequality that $\langle \phi,
    (z-\Delta)^{-1} M_{\chi} \psi\rangle$ is $C^{k,\alpha}$ on $(0,+\infty)+i[0,+\infty)$.
\end{proof}

For $\lambda > 0$, we denote by
\begin{align*}
  (\lambda+i0^{+} + \Delta)^{-1} = \lim_{\eta \to 0^{+}} (\lambda+i\eta + \Delta)^{-1}
\end{align*}
the boundary value of the free resolvent. Its action can be explicitly
computed using the spectral representation \eqref{eq:spectral_repr}
and the Plemelj-Sokhotski formula \eqref{eq:plemelj}. Note in
particular that it differs from $(\lambda+i0^{-} + \Delta)^{-1}$ by
the sign of its anti-hermitian part.

\begin{proposition}[Limiting absorption principle for $H= - \Delta+V$]
\label{prop:LAP_pop}

Let $s = \tfrac 1 2 + k + \alpha$ for $k \in \N$ and $\alpha \in
[0,1]$. Let $V : \R^{d} \to \R$ be a continuous potential
such that $\langle x \rangle^{2s+\epsilon} V$ is bounded, for some
$\varepsilon > 0$. The resolvent $(z-H)^{-1}$ defined for
${\rm Im }z > 0$ extends to an operator of class
$C^{k,\alpha}$ on
the semi-open set $(0,+\infty) + i [0,+\infty)$, in the topology of
bounded operators from $L^{2}\left(\japx^{s}\right)$ to
$H^{2}\left(\japx^{-s}\right)$.
\end{proposition}

\begin{proof}
  
We use the following resolvent inequality:
\begin{equation}
  \label{eq:LAP_V_res_eq}
  (z-H)^{-1} = B(z)^{-1} (z+\Delta)^{-1}
\end{equation}
with
\begin{align*}
  B(z) = 1-(z+\Delta)^{-1} V,
\end{align*}
valid for $z \in \C$ with ${\rm Im }z > 0$. Since $V$ is bounded from
$H^{2}(\japx^{-s})$ to $L^{2}(\japx^{s})$, it follows from Proposition
\ref{prop:LAP_free} that $B(z)$ extends to an operator of class
$C^{k,\alpha}$ on the semi-open set $(0,+\infty) + i [0,+\infty)$, in
the topology of bounded operators on $H^{2}\left(\japx^{-s}\right)$.

We will show that for all $\lambda > 0$, $B(\lambda+i0^{+})$ is
invertible on $H^{2}\left(\japx^{-s}\right)$. This shows that
$B(z)^{-1}$ is $C^{k,\alpha}$ on the semi-open set in the topology of
bounded operators on $H^{2}\left(\japx^{-s}\right)$, which implies our
result by \eqref{eq:LAP_V_res_eq} and Proposition \ref{prop:LAP_free}.

Let $\lambda > 0$. Since $\japx^{2s+\varepsilon} V$ is bounded, the
multiplication operator $V$ is compact from $H^2(\langle x
\rangle^{-s})$ to $L^2(\langle x \rangle^s)$. It follows that
$(\lambda+i0^{+}+\Delta)^{-1} V $ is compact on
$H^{2}\left(\japx^{-s}\right)$. By the Fredholm alternative, it is then
enough to show that there are no non-zero solutions of
\begin{align*}
  u = (\lambda+i0^{+}+\Delta)^{-1} V u
\end{align*}
in $H^{2}\left(\japx^{-s}\right)$. Let $u \in
H^{2}\left(\japx^{-s}\right)$ be such a non-zero solution. Testing this equality
against $V u$ and taking imaginary parts, we obtain from the
Plemelj-Sokhotski formula \eqref{eq:plemelj} that
\begin{align*}
  0 = \Im( \langle Vu , u \rangle ) = {\rm Im}\langle  Vu, (\lambda+i0^{+}+\Delta)^{-1} V u \rangle = -\frac{\pi}{2\sqrt{\lambda}} \int_{|q|=\sqrt{\lambda}} |\mathcal{F}(Vu)(q)|^2 \, \mathrm{d}q.
\end{align*}
By Lemma~\ref{lem:trace_hardy}, 
\begin{align*}
  \mathcal{F}u(q) = \frac{\mathcal{F}(Vu)(q)}{\lambda - |q|^2}
\end{align*}
with $\cF(V u) \in H^{s}(\R^{d})$ shows that $\langle  q \rangle^{2} \cF u(q) \in H^{s-1}(\R^{d})$, and so that $u
\in H^{2}(\japx^{s-1})$. More generally, the argument above shows that
if $u \in L^{2}(\japx^{s'})$ with $s' \ge -s$, then $u \in H^{2}(\japx^{s'+2s-1})$, and,
since $s > \tfrac 1 2$, it follows that $u \in H^{2}(\R^{d})$, and
therefore that $\lambda$ is a positive embedded eigenvalue, which is impossible.
\end{proof}

\begin{proof}[Proof of Theorem~\ref{thm:main1}]
  By Theorem~\ref{thm:Kubo}, the response function is given by
  \begin{equation}
         \widehat K(\omega) = \lim_{\eta \to 0^{+}} \lela \psi_{0}, V_{\mathcal O} \Big(\omega +i\eta - (H-E_0)\Big)^{-1} V_{\mathcal P} \psi_{0}\rira - \lela \psi_{0},V_{\mathcal P} \Big(\omega +i\eta + (H-E_0)\Big)^{-1} V_{\mathcal O} \psi_{0}\rira.
  \end{equation}
  By the exponential localization of the ground state wave function
  $\psi_0$, and the assumptions on the potentials $V_{\mathcal O}$ and
  $V_{\mathcal P}$, $V_{\mathcal O}\psi_0$ and $V_{\mathcal P}\psi_0$
  belong to every $L^2(\japx^s)$ for $s \in \R$. 
  Since by Assumption~\ref{assump:decay_V} the function
  $\japx^{2+\epsilon} V$ is bounded, the result follows by Proposition
  \ref{prop:LAP_pop} in the case $k=0, \alpha=1$.
\end{proof}

\section{Truncation in space}
\label{sec:truncation}
Consider the domain $\Omega_{L} = [-L,L]^{d}$ with Dirichlet boundary
conditions. We define $\widetilde{H}_{L}$ the operator $-\Delta + V$
with domain
$D(\widetilde H_{L}) = \{\widetilde \psi \in H^{2}(\Omega_{L}),
\widetilde \psi|_{\partial \Omega_{L}} = 0\}$, self-adjoint on
$L^{2}(\Omega_{L})$. This operator is bounded from below and has
compact resolvent.

We now define the operator $H_{L}$ on $L^{2}(\R^{d})$ in the following
way: if $\psi \in L^{2}(\R^{d})$ and $\psi|_{\Omega_{L}} \in D(\widetilde H_{L})$,
then
\begin{align*}
  ({H_{L}} \psi)|_{\Omega_{L}} = \widetilde H_{L} \psi|_{\Omega_{L}},
\end{align*}
and $({H_{L}} \psi)|_{\R^{d} \setminus \Omega_{L}} = 0$. This defines an operator on
$L^{2}(\R^{d})$, self-adjoint with domain
$D(H_{L}) = L^{2}(\R^{d} \setminus \Omega_{L}) \oplus D(\widetilde H_{L})$, and with
spectrum $\sigma(\widetilde H_{L}) \cup \{0\}$. Let $\psi_{0,L}$ be an
$L^2$-normalized eigenvector associated to the lowest eigenvalue of
$H_L$.

Note that by adapting the proof in Lemma \ref{lem:growth-dynamics},
the estimates shown there for $e^{-i t H}$ on $L^{2}(\R^{d})$ are also
valid for $e^{-i t \widetilde H_{L}}$ on $L^{2}(\Omega_{L})$, with
constants independent of $L$.
Similarly, the estimates of
Lemma~\ref{lem:resolvent} for $(z-H)^{-1}$ on $L^{2}(\R^{d})$ and
$L^{2}(e^{\alpha \japx})$ are also valid for
$(z-\widetilde H_{L})^{-1}$ on $L^{2}(\Omega_{L})$ and
$L^{2}(e^{\alpha \japx}; \Omega_{L}) = \{\psi, e^{\alpha \japx} \psi
\in L^{2}(\Omega_{L})\}$ with natural norms, still with constants
independent of $L$.

We can now define $K_{L}$ analogously to $K$:
  \begin{align}
    \label{eq:kubo_time_L}
    K_{L}(\tau) &= -i \theta(\tau) \lela V_{\mathcal O} \psi_{0,L},  e^{-i (H_L-E_{0,L}) \tau} V_{\mathcal P} \psi_{0,L}\rira + {\rm c.c.},
  \end{align}
and
  \begin{align}
    \label{eq:K_frequency_L}
    \widehat K_{L}(\omega) = \lim_{\eta \to 0^{+}} &\lela \psi_{0,L}, V_{\mathcal O} \Big(\omega +i\eta - (H_{L}-E_{0,L})\Big)^{-1} V_{\mathcal P} \psi_{0,L}\rira\\\notag -& \lela \psi_{0,L},V_{\mathcal P} \Big(
    \omega +i\eta + (H_{L}-E_{0,L})\Big)^{-1} V_{\mathcal O} \psi_{0,L}\rira.
  \end{align}
  The operator $H$ and $H_{L}$ have the same action, but $H$ has domain
  $D(H) = H^{2}(\R^{d})$, whereas $H_{L}$ has domain
  $D(H_{L}) = \{\psi \in L^{2}(\R^{d}), \psi|_{\Omega_{L}} \in
  H^{2}(\Omega_{L}), \psi|_{\partial \Omega_{L}} = 0\}$. These different
  domains do not even share a common core, making the direct
  comparison of $K_{L}$ and $K$ difficult. However, we will prove and
  use the fact that, when evaluated on localized quantities, their
  resolvents
  \begin{align}
    R(z) = (z-H)^{-1}, \quad R_{L}(z) = (z-H_{L})^{-1}
  \end{align}
  and propagators $e^{-i H t}$ and $e^{-i H_{L} t}$, both defined on
  $L^{2}(\R^{d})$, are close. To that end, we let $\chi : \R^{d} \to \R$ be a smooth truncation
  function equal to $1$ for $|x|_{\infty} \le 1/4$ and to 0 for
  $|x|_{\infty} \ge 3/4$, and
\begin{align*}
  \chi_{L}(x) = \chi(x/L).
\end{align*}
Note that, as a multiplication operator, $\chi_{L}$ maps
$D(H) \cup D(H_{L})$ to $D(H) \cap D(H_{L})$.

Furthermore, this truncation is
exponentially accurate on exponentially localized functions: by direct
computation, for all $k \in \N$ there is $C_{k} > 0$ such that, for
all $0 \le \alpha_{1} \le \alpha_{2} \le 1$, for all
$\psi \in H^{k}(e^{\alpha_{2} \japx})$
\begin{align*}
  \|(1-\chi_{L})\psi\|_{H^{k}(e^{\alpha_{1} \japx})} = \|e^{(\alpha_{1}-\alpha_{2}) \japx} (1-\chi_{L}) e^{\alpha_{2} \japx} \psi\|_{H^{k}(\R^{d})} \le C_{k} e^{-(\alpha_{2}-\alpha_{1}) L} \|\psi\|_{H^{k}(e^{\alpha_{2} \japx})}.
\end{align*}

\begin{lemma}
  \label{lem:resolvent_convergence}
  There are $c > 0, C > 0$ such that, for all $z \in \mathbb C$ and
  $L > 0$ such that $d(K, \sigma(H)) \ge g$ and
  $\liminf d(K, \sigma(H_{L})) \ge g$ with $g > 0$, for all
  $0 \le \alpha \le \alpha' \le c g$,
  \begin{align*}
    \|R_{L}(z) - R(z)\|_{L^{2}(e^{\alpha' \japx}) \to L^{2}(e^{\alpha \japx})} \le C \left( 1+\frac 1 {g}\right)^{2} (1+|z|)^{3} e^{-(\alpha'-\alpha)L}.
  \end{align*}
\end{lemma}
\begin{proof}
  
  Because of the aforementioned domain issues, we cannot directly use the resolvent
  formula $R(z) - R_{L}(z) = R_{L}(z) (H - H_{L}) R(z)$. However, we can approximate any $\psi \in L^{2}(\R^{d})$ by
  $R(z)^{-1} \chi_{L} R(z) \psi$, for which
  \begin{align*}
    (R(z) - R_{L}(z)) \, R(z)^{-1} \chi_{L} R(z) \psi = R_{L}(z) (H-H_{L}) \chi_{L} R(z) \psi = 0,
  \end{align*}
  where we have used that $H \phi = H_{L} \phi$ for all
  $\phi \in D(H_{L}) \cap D(H)$. Therefore, using the estimates of Lemma
  \ref{lem:resolvent} for both $H$ and $H_{L}$,
  \begin{align*}
    \|(R(z) - R_{L}(z)) \psi\|_{L^{2}(e^{\alpha \japx})} &= \|(R(z)-R_{L}(z)) (R(z)^{-1} (1-\chi_{L}) R(z) \psi)\|_{L^{2}(e^{\alpha \japx})}\\
    &\lesssim \left( 1+ \frac 1 {d(z, \sigma(H))} + \frac 1 {d(z, \sigma(H_{L}))}\right)(1+|z|) \|R(z)^{-1} (1-\chi_{L}) R(z) \psi\|_{L^{2}(e^{\alpha \japx})}\\
    &\lesssim \left( 1+\frac 1 {d(z, \sigma(H))} + \frac 1 {d(z, \sigma(H_{L}))}\right) (1+|z|)^{2} \|(1-\chi_{L}) R(z) \psi\|_{H^{2}(e^{\alpha \japx})}\\
    &\lesssim \left( 1+\frac 1 {d(z, \sigma(H))} + \frac 1 {d(z, \sigma(H_{L}))}\right)^{2} (1+|z|)^{3} e^{-(\alpha'-\alpha)L} \|\psi\|_{L^{2}(e^{\alpha' \japx})}.
  \end{align*}
\end{proof}

Using this we can compare the eigenpairs of $H_{L}$ and $H$.
\begin{lemma}
  There are $C, \alpha_{1}, \alpha_{2} > 0$ such that, for all $L$ large enough,
  \begin{align}
    \label{eq:conv_exp_E}
    |E_{0,L} - E_{0}| &\le  C e^{-\alpha_{0} L}\\
    \label{eq:conv_exp_psi}
    \|{\psi_{0,L}} - \psi_{0}\|_{L^{2}(e^{\alpha_{1} \japx})} &\le C e^{-\alpha_{2} L}
  \end{align}
  where the sign of $\psi_{0,L}$ is chosen such that $\langle
  {\psi_{0,L}}, \psi_{0} \rangle\ge 0$, and where $\alpha_{0}$ is the
  constant in Lemma \ref{lem:psi0_exp_loc}.
\end{lemma}
\begin{proof}

  Since by Lemma \ref{lem:psi0_exp_loc} $\psi_{0} \in H^{2}(e^{\alpha_{0} \japx})$,
  \begin{align*}
    \|(1-\chi_{L})\psi_{0}\|_{H^{2}(\R^{d})} \lesssim e^{-\alpha_{0} L}.
  \end{align*}
  and \eqref{eq:conv_exp_E} follows from the variational principle
  \begin{align*}
    E_{0} \le E_{0,L} \le \frac{\langle\chi_{L} \psi_{0}, H \chi_{L} \psi_{0}\rangle_{L^{2}(\R^{d})}}{\langle\chi_{L} \psi_{0}, \chi_{L} \psi_{0}\rangle_{L^{2}(\R^{d})}} \le E_{0} + C e^{-\alpha_{0} L}.
  \end{align*}
  for some $C > 0$.

  Let $E_{1,L}$ and $E_{1}$ be the second-lowest eigenvalue (or zero if
  there are no second eigenvalue) of $H_{L}$ and $H$ respectively. From the
  min-max principle, $E_{1,L} \ge E_{1}$ and therefore for $L$ large
  enough there is a gap $g > 0$ in $\sigma(H_{L})$ above $E_{0,L}$.
  Let $C$ be the circle with center $E_{0}$ and radius $g/2$ in the
  complex plane, oriented trigonometrically. Then, by Lemma
  \ref{lem:resolvent_convergence} there is
  $\alpha_{2} > 0$ such that
  \begin{align*}
    1 - \langle \psi_{0}, \psi_{0,L} \rangle^{2} &= \left\langle  \psi_{0}, \Big( |\psi_{0}\rangle\langle \psi_{0}| -|\psi_{0,L}\rangle\langle \psi_{0,L}|  \Big)\psi_{0} \right\rangle\\
    &= \frac 1 {2\pi i} \oint_{\mathcal C} \left\langle  \psi_{0}, (R(z) - R_{L}(z))\psi_{0} \right\rangle dz\\
    |1 - \langle \psi_{0}, \psi_{0,L} \rangle^{2}| &\lesssim e^{-\alpha_{2} L}
  \end{align*}
  Then
  \begin{align*}
    \frac 1 2 \|\psi_{0} - \psi_{0,L}\|^{2}_{L^{2}(\R^{d})} = 1 - \langle\psi_{0},\psi_{0,L}\rangle = 1 - \sqrt{\langle\psi_{0},\psi_{0,L}\rangle^{2}} \lesssim e^{-\alpha_{2} L}.
  \end{align*}

  Now, as in Lemma \ref{lem:psi0_exp_loc}, let
  $V = V_{\varepsilon} + V_{c}$ with $V_{c}$ compactly supported and
  $\|V_{\varepsilon}\|_{L^{\infty}(\R^{d})} \le -E_{0}/2$. Let
  $H_{L,\varepsilon} = -\Delta + V_{\varepsilon}$ on $\Omega_{L}$,
  extended as before to act on $L^{2}(\Omega)$. For
  $L$ large enough so that the support of $V_{c}$ is contained in
  $\Omega_{L}$, we have
  \begin{align*}
    \psi_{0,L} &= (E_{0,L}-H_{L,\varepsilon})^{-1} V_{c} \psi_{0,L}.
  \end{align*}
  Arguing as in Lemma \ref{lem:resolvent_convergence}, there are
  $\alpha_{1},\alpha' > 0$ such that
  $(E_{0,L}-H_{L,\varepsilon})^{-1}$ converges exponentially quickly
  to $(E_{0}-H)^{-1}$ as an operator from $L^{2}(e^{\alpha' \japx})$ to
  $L^{2}(e^{\alpha_{1} \japx})$. Furthermore, because $V_{c}$ is compactly supported, we have
  \begin{align*}
    \|V_{c} \psi_{0,L} - V_{c} \psi_{0}\|_{L^{2}(e^{\alpha' \japx})} \lesssim \|\psi_{0,L} - \psi_{0}\|_{L^{2}(\R^{d})} \lesssim e^{-\alpha_{2} L}
  \end{align*}
  and the result follows.
\end{proof}

With this we can now prove the convergence of $K_{L}(\omega+i\eta)$ for
positive $\eta$.
\begin{theorem}
  There are $\alpha_{3} > 0, C > 0$ such that for all
  $0 < \eta < 1$, $\omega \in \R$,
  \begin{align*}
    |\widehat K_{L}(\omega + i\eta) - \widehat K(\omega + i\eta)| \le \frac{C(1+\omega)^{3}}{\eta^{2}}e^{-\alpha_{3} \eta L}
  \end{align*}
  
\end{theorem}
\begin{proof}
  Since $\psi_{0,L}$ converges exponentially towards $\psi_{0}$ in
  $L^{2}(e^{\alpha \japx})$ for some $\alpha > 0$, and $V_{\mathcal O}$
  and $V_{\mathcal P} $ have at most polynomial growth, $V_{\mathcal O}  \psi_{0,L}$
  and $V_{\mathcal P} \psi_{0,L}$ converge exponentially quickly in
  $L^{2}(\R^{d})$ towards $V_{\mathcal O}  \psi_{0}$ and $V_{\mathcal P} \psi_{0}$
  respectively. $E_{0,L}$ converges exponentially towards $E_{0}$ and
  $\Big(\omega +i\eta - (H_{L}-E_{0,L})\Big)^{-1}$ is uniformly
  bounded by $1/\eta$ as an operator on $L^{2}(\R^{d})$. It therefore
  follows that we can reduce to terms of the
  form
\begin{align*}
  \left\langle  \psi_{l}, \Big(R_{L}(\omega-E_{0}+i\eta) - R(\omega-E_{0}+i\eta)\Big) \psi_{r}\right\rangle
\end{align*}
with $\psi_{l/r} \in L^{2}(e^{\alpha \japx})$ for some $\alpha > 0$
independent on $\omega, \eta$. We can then conclude using Lemma \ref{lem:resolvent_convergence}.
\end{proof}

Finally, we conclude the proof of Theorem \ref{thm:main_2}.
\begin{theorem}
  $K_{L}$ converges as a tempered distribution towards $K$.
\end{theorem}
\begin{proof}
  We will prove that, for all $f \in \mathcal S(\R)$,
  \begin{align*}
    \int_{0}^{\infty} \lela V_{\mathcal O} \psi_{0,L}, e^{-i(H_{L}-E_{0,L})t} V_{\mathcal P} \psi_{0,L}\rira f(t) dt \rightarrow \int_{0}^{\infty} \lela V_{\mathcal O} \psi_{0}, e^{-i(H-E_{0})t} V_{\mathcal P} \psi_{0}\rira f(t) dt.
  \end{align*}
  Since $\psi_{0,L} \to \psi_{0}$ in $H^{2}(e^{\alpha_{1} \japx})$,
  \begin{align*}
    \|V_{\mathcal O} \psi_{0,L} - \chi_{L} V_{\mathcal O} \psi_{0}\| \to 0
  \end{align*}
  and similarly for $V_{\mathcal P}$. It is therefore sufficient to prove that
  \begin{align*}
    \|\chi_{L} (e^{-iHt} - e^{-iH_{L}t}) \chi_{L} V_{\mathcal P} \psi_{0}\| \le \frac{P(t)}{L}
  \end{align*}
  for some polynomial $P$. Let
  \begin{align*}
    \phi(t) = e^{-iH t } \chi_{L} V_{\mathcal P} \psi_{0}, \quad \phi_{L}(t) = e^{-iH_{L} t } \chi_{L} V_{\mathcal P} \psi_{0}
  \end{align*}
  To estimate $\chi_{L}(\phi(t) - \phi_{L}(t))$ we compute
  \begin{align*}
    i\partial_{t} \chi_{L}(\phi - \phi_{L}) &= H \chi_{L} \phi - H_{L} \chi_{L} \phi_{L} + [\chi_{L}, H]\phi - [\chi_{L}, H_{L}] \phi_{L}\\
    &= H \chi_{L}(\phi - \phi_{L}) + [\chi_{L}, H]\phi - [\chi_{L}, H_{L}] \phi_{L}
  \end{align*}
  and therefore by the Duhamel formula
  \begin{align*}
    \chi_{L}(\phi(t)-\phi_{L}(t)) &= -i \int_{0}^{t} e^{-iH(t-t')} \left( [\chi_{L}, H] \phi(t') - [\chi_{L}, H_{L}] \phi_{L}(t') \right) dt'\\
    \|\chi_{L}(\phi(t)-\phi_{L}(t))\| &\le t \sup_{t' \in [0,t]}\left\| [\chi_{L}, H] \phi(t') - [\chi_{L}, H_{L}] \phi_{L}(t') \right\|.
  \end{align*}
  Since
  $[\chi_{L}, H] \phi = 2 \nabla\chi_{L} \cdot \nabla\phi + \Delta \chi_{L}\phi$ is
  zero for $|x|_{\infty} < L/4$, by Lemma~\ref{lem:growth-dynamics} we have
  \begin{align*}
    \|[\chi_{L}, H] \phi(t')\| &\lesssim \frac 1 L \|(1+|x|)[\chi_{L}, H] \phi(t')\|\\
    &\lesssim \frac 1 L  \left( \|x \otimes \nabla \phi(t')\| + \|\nabla \phi(t')\| + \|(1+|x|)\phi(t')\|\right)\\
    &\lesssim \frac{1+|t|^{4}}{L}
  \end{align*}
  and similarly with $[\chi_{L},H_{L}] \phi_{L}(t')$. The result follows.
\end{proof}

\section{Appendix: trace theory in Sobolev spaces}
We will need the following lemma on the regularity of traces on
surfaces with respect to variations of the surface.
\begin{lemma}
  \label{lem:trace}
  Let $\chi : \R \to \R$ be a smooth function with
  support $[R_{1}, R_{2}]$, with $R_{1} > 0$, and $s_{1}, s_{2}, s$
  nonnegative real numbers such that $s_{1} + s_{2} = s$. Then for all $u \in
  H^{s}(\R^{d})$, the function
  \begin{align*}
    r \mapsto (\hat x \mapsto \chi(r) u(r \hat x))
  \end{align*}
  is in $H^{s_{1}}(\R, H^{s_{2}}(S^{d-1}))$.
\end{lemma}
\begin{proof}
  We first treat the case of the restriction to a hyperplane: if $u \in H^{s}(\R^{d})$, then
  \begin{align*}
    v_{u} : x_{1} \mapsto (x' \mapsto u(x_{1}, x'))
  \end{align*}
  is in $H^{s_{1}}(\R, H^{s_{2}}(\R^{d-1}))$. Indeed, denoting for clarity by $\mathcal F_{1}$
  the one-dimensional Fourier transform, we have by the Parseval
  formula on $L^{2}(\R^{d-1})$ that for all $q_{1} \in \R$,
  \begin{align*}
    \|\cF_{1} v_{u}(q_{1})\|_{H^{s_{2}}(\R^{d-1})}^{2} = \frac {1}{(2\pi)^{d-1}}\int_{\R^{d-1}} \langle q' \rangle^{2s_{2}}|\cF u(q_{1}, q')|^{2} \mathrm{d}q'
  \end{align*}
  and therefore
  \begin{align*}
    \|v_{u}\|^{2}_{H^{s_{1}}(\R, H^{s_{2}}(\R^{d-1}))} = \frac 1 {2\pi} \int_{\R} \langle q_{1}\rangle^{2s_{1}} \|\cF_{1} v_{u}(q_{1})\|_{H^{s_{2}}(\R^{d-1})}^{2} \mathrm{d}q_{1} \le \frac 1 {(2\pi)^{d}}\int_{\R^{d}} \langle  q \rangle^{2s}|\cF u(q)|^{2}  \mathrm{d}q = \|u\|_{H^{s}(\R^{d})}^{2}.
  \end{align*}
  To lift this property to the sphere $S^{d-1}(R)$ of radius $R$, we
  use a classical ``flattening'' argument. Using spherical
  coordinates, we can construct a cover of the annulus of inner radius
  $R_{1}$ and outer radius $R_{2}$ by open sets
  $\{X_{i}\}_{i=1,\dots,N}$ not touching zero with the property that,
  for every $i \in \{1, \dots, N\}$, there is a smooth diffeomorphism
  $\Phi_{i}$ from a an open set $\mathcal R_{i} \times \mathcal
  T_{i} \subset \R \times \R^{d-1}$ to
  $X_{i}$ such that, for all $(r,\theta) \in \mathcal R_{i} \times \mathcal
  T_{i}$,
  \begin{align*}
    \Phi_{i}(r,\theta) = r\, \Theta_{i}(\theta)
  \end{align*}
  with $\Theta_{i}$ having values on the sphere $S^{d-1}$.
  One can then construct a partition of the unity
  $\zeta_{i} : \R^{d} \to \R$ where, for each $i \in \{1, \dots, N\}$,
  $\zeta_{i}$ is supported inside $X_{i}$, and
  $\sum_{i=1}^{N} \zeta_{i} = 1$ on the annulus. Then, for all $u \in
  H^{s}(\R^{d})$, $|x| \in [R_{1},R_{2}]$,
  \begin{align*}
    \chi(|x|) u(x) = \sum_{i=1}^{N} \chi(|x|)\, \zeta_{i}(x)\, u(x) = \sum_{i, \; x \in X_{i}} w_{i}(\Phi_{i}^{-1}(x))
  \end{align*}
  where
  \begin{align*}
    w_{i}(r,\theta) = \chi(r)\, \zeta_{i}(r\Theta_{i}(\theta))\, u(r \Theta_{i}(\theta)),
  \end{align*}
  defined on $\mathcal R_{i} \times \mathcal
  T_{i}$, extends on the whole $\R^{d}$ to a  $H^{s}(\R^{d})$
  function. It follows from the hyperplane case that
  \begin{align*}
    \chi(r) u(r \hat x) = \sum_{i, \; x \in X_{i}} w_{i}\left(r, \Theta_{i}^{-1}(\hat x)\right)
  \end{align*}
  is in $H^{s_{1}}(\R, H^{s_{2}}(S^{d-1}))$.
\end{proof}
Note that from the fact that $H^{\frac 1 2 +\varepsilon}(\R)$, $\varepsilon>0$, functions are continuous, we
recover the classical trace theorem that traces of $H^{s+\frac 1 2+\varepsilon}(\R^{d})$
functions are $H^{s}$ on surfaces.

The proof of the limiting absorption principle for the nonzero
potential case requires the following Hardy-type inequality.
\begin{lemma}
    \label{lem:trace_hardy}
    Let $s>0$, and $u \in H^s(\R^d)$ such that $u$ is zero on
    the sphere of radius $a$ (in the sense of traces). Then the function
    $v(x) = \tfrac{u(x)}{|x|^{2}-a^{2}}$ is $H^{s-1}(\R^d)$.
\end{lemma}

\begin{proof}
  Using as before a smooth cutoff function and a partition of unity of a neighborhood of the sphere
  of radius $a$, it is enough to show that for $u\in H^s(\R^d)$
  with $u(0,x')=0$ for all $x' \in \R^{d-1}$, then
  $v(x) = \tfrac{u(x)}{x_{1}}$ is $H^{s-1}(\R^{d-1})$.
  
Proceeding by density, we can assume that $\mathcal F u \in
  C^{\infty}_{c}(\R^{d})$. Then, using the fact that $\int \mathcal F u
  (q_{1}, q') \mathrm{d}q_1 = 0$ for all $q' \in \R^{d-1}$, we have that $\mathcal F
  v  \in C^{\infty}_{c}(\R^{d})$, and  $\mathcal{F}u(q) = -i \tfrac{\partial \mathcal{F}v}{\partial
    q_1}(q)$. By integration by parts and the Cauchy-Schwarz
  inequality, we have the following Hardy inequality, for $q = (q_{1},
  q')$ and $\alpha \in \R$:
  \begin{equation*}
    \int_\R |\mathcal{F}v(q)|^2 \langle q_1\rangle^{2(\alpha-1)} \, \mathrm{d}q_1 \lesssim \int_\R |\mathcal{F}v(q)| |\mathcal F u(q)| \langle q_1\rangle^{(\alpha-1)+\alpha} \, \mathrm{d}q_1 \lesssim \int_\R {|\mathcal F u(q)|}^2 \langle q_1\rangle^{2\alpha}\, \mathrm{d}q_1.
  \end{equation*}
In the case $s \geq 1$, we have $\langle  q \rangle^{2(s-1)} \lesssim
\langle q_1\rangle^{2(s-1)} + \langle q' \rangle^{2(s-1)}$ and so
  \begin{align*}
      \|v\|^2_{H^{s-1}} &\lesssim \int_{\R^{d-1}} \int_\R |\mathcal{F}v(q)|^2 \left(\langle q_1\rangle^{2(s-1)} + \langle q' \rangle^{2(s-1)}\right) \, \mathrm{d}q_1 \mathrm{d}q' .
    \end{align*}
    By using the Hardy inequality with $\alpha = s$ for the first
    term and $\alpha = 1$ for the second, we get
    \begin{align*}
            \|v\|^2_{H^{s-1}} &\lesssim \int_{\R^d} |\mathcal F u(q)|^{2} \langle q_1 \rangle^{2s}\, \mathrm{d}q + \int_{\R^d} |\mathcal F u(q)|^{2} \langle q_1\rangle^{2}\langle q' \rangle^{2(s-1)}\, \mathrm{d}q \\
      &\lesssim \int_{\R^d} |\mathcal F u(q)|^{2} \langle q \rangle^{2s}\, \mathrm{d}q  \lesssim \|u\|_{H^s}^2.
  \end{align*}
  In the case $0 < s < 1$, we have $\langle  q \rangle^{2(s-1)} \le
\langle q_1\rangle^{2(s-1)}$ and we can repeat the above argument.
\end{proof}

\section{Appendix: locality estimates on resolvents and propagators}
We prove in this appendix results on the locality of the resolvents
and propagators of Schrödinger operators. This appendix is independent
from the rest of the paper.
\begin{lemma}[Properties of the propagator]
  \label{lem:growth-dynamics}
  Let $W : \R \times \R^{d} \to \R$ be a potential such that $W$ is
  continuous on $\R \times \R^{d}$ and for all $t \in \R$,
  $W(t,\cdot)$ is $C^{\infty}(\R^{d})$ and satisfies the sub-linear
  condition: for all $|\alpha| \ge 1$, $\partial_{x}^{\alpha} W$ is
  bounded on $\R \times \R^{d}$.

  There exists a unitary propagator $U(t,s)$ such that if
  $\psi_{s} \in \mathcal S(\R^{d})$, $U(t,s) \psi_{s} \in \mathcal S(\R^{d})$ satisfies
  the Schrödinger equation
  \begin{align*}
    i\partial_{t} U(t,s) \psi_{s} = (-\Delta + W(t))U(t,s) \psi_{s}.
  \end{align*}

  Furthermore, there is $C_{0} > 0$ not depending on $W$ such
  that, for all $t, s \in \R, \psi_{s} \in L^{2}(\R^{d})$,
    \begin{align}
      \label{eq:control_U1}
      \|xU(t,s) \psi_{s}\| + \|\nabla U(t,s)\psi_{s}\| &\le C_{1}(t,s) \left(\|x\psi_{s}\| + \|\nabla \psi_{s}\| + \|\psi_{s}\|\right)\\
      \label{eq:control_U2}
      \||x|^{2} U(t,s) \psi_{s}\| + \|\Delta U(t,s)\psi_{s}\| &\le C_{2}(t,s) \left(\||x|^{2}  \psi_{s}\| + \|\Delta \psi_{s}\| + \|x \otimes \nabla \psi_{s}\| + \|\psi_{s}\|\right)
    \end{align}
    where
    \begin{align*}
        C_{1}(t,s) &= C_{0} (1 + |t-s|)^{2}\left(1 + \sup_{t' \in [t,s]} |\nabla W(t')|\right)\\
        C_{2}(t,s) &= C_{0} (1 + |t-s|)^{4}\left(1 + \sup_{t' \in [t,s]} |\nabla W(t')|^{2} + \sup_{t' \in [t,s]} |\nabla^{2}W(t')|\right)
    \end{align*}
  \end{lemma}
  Note that these estimates are natural in the case $W=0$. In this
  case, $(\cF \psi)(t,q) = (\cF \psi)(s,q) e^{-i (t-s) |q|^{2}}$, and so $\partial_{q}
  (\cF \psi)(t,q) = \partial_{q} (\cF \psi)(s,q) e^{-i |q|^{2}(t-s)} - 2i
  q(t-s)(\cF
  \psi)(s,q) e^{-i |q|^{2}(t-s)}$, which is in $L^{2}(\R^{d})$ if $\psi_{s}
  \in L^{2}(\japx) \cap H^{1}(\R^{d})$.
  \begin{proof}
    The existence of the propagator is obtained using the results of
    \cite{fujiwara1979construction} (which actually only requires a
    sub-quadratic potential).
    
    We will obtain these inequalities by the following standard
    commutator method. Let $A$
    be an operator, and $\psi_{s} \in L^{2}(\R^{d})$. Then, if
    $\psi(t) = U(t,s) \psi_{s}$, we have
  \begin{align*}
    i\partial_{t} (A \psi)(t') = A H(t') \psi(t') = H(t') A\psi(t') + [A,H(t')] \psi(t')
  \end{align*}
  and therefore by Duhamel's formula,
  \begin{align*}
    A \psi(t) = U(t,s) A\psi(s) - i \int_{s}^{t} U(t,t') [A,H(t')] U(t',s)  dt'.
  \end{align*}
  
  We compute the following commutators
  \begin{align*}
    [\nabla, H(t)] &= \nabla W(t)\\
    [x, H(t)] &= 2\nabla
  \end{align*}
  Since $\nabla W$ is a bounded operator, we obtain
  \begin{align*}
    \|\nabla U(t,s)\psi_{s}\| &\le \|\nabla\psi_{s}\| + |t-s| \sup_{t' \in [s,t]} |\nabla W(t')| \|\psi_{s}\|\\
    \|x U(t,s) \psi_{s}\| &\le \|x \psi_{s}\| + 2 |t-s| \sup_{t' \in [s,t]} \|\nabla U(t',s)\psi_{s}\|\\
    &\le \|x \psi_{s}\| + 2 |t-s| \left(\|\nabla\psi_{s}\| + |t-s| \sup_{t' \in [s,t]} |\nabla W(t')| \|\psi_{s}\|\right)
  \end{align*}
  and \eqref{eq:control_U1} follows. Similarly, from the commutators
  \begin{align*}
    [\nabla^{2}, H(t)] &= \nabla^{2}W + 2(\nabla W(t)) \otimes \nabla\\
    [x\otimes \nabla, H(t)] &= -2 \nabla^{2} + x \otimes \nabla W(t)\\
    [x \otimes x, H(t)] &= 2x \otimes \nabla + 2 I
  \end{align*}
  we obtain \eqref{eq:control_U2}.
\end{proof}

\begin{lemma}[Properties of the resolvent]
  \label{lem:resolvent}
  Let $W : \R^{d} \to \R$ be a bounded function, and $R(z) = (z-(-\Delta+W))^{-1}$. Then there are $c >0, C > 0$ such that, for all $z \notin \sigma(H)$,
  \begin{align*}
    \|R(z)\|_{L^{2}(\R^{d}) \to H^{2}(\R^{d})} &\le C(1+|z|)\left(1+\frac {1}{d(z, \sigma(H))}\right)\\
    \|R(z)\|_{L^{2}(e^{\alpha\japx}) \to H^{2}(e^{\alpha\japx})} &\le C(1+|z|)\left(1+\frac {1}{d(z, \sigma(H))}\right) \quad \forall \alpha \le \alpha_{z} := c d(z, \sigma(H))
  \end{align*}
\end{lemma}
\begin{proof}
  The first inequality is classical (see for instance
  \cite{levitt2020screening} Lemma 3.6).

  The second is a (non-sharp) Combes-Thomas estimate, which we prove for
  completeness here. Denote by
  \begin{align}
    \label{eq:Halpha}
      H_{\alpha} := e^{\alpha \langle x \rangle} (-\Delta+W) e^{-\alpha \langle x \rangle} = (-\Delta+W) + \underbrace{-2\alpha \nabla\langle x \rangle \cdot \nabla + \alpha^{2} \Delta(\langle x \rangle)}_{\alpha B_{\alpha}}.
    \end{align}
    Let $R(z) = (z-(-\Delta+W))^{-1}$. We have that
    \begin{align*}
      B_{\alpha} R(z)  =(-2 \nabla\langle x \rangle \cdot \nabla + \alpha \Delta(\langle x \rangle))  (1-\Delta)^{-1} (1-\Delta)R(z)
    \end{align*}
    is bounded as an operator on $L^{2}(\R^{d})$ by $C/d(z, \sigma(H))$,
    for all $\alpha \le 1$, for some $C > 0$. It follows that, for $\alpha \le d(z,\sigma(H))/(2C)$
    \begin{align*}
      (z-H_{\alpha})^{-1} = R(z)(1 + \alpha B_{\alpha} R(z))^{-1}
    \end{align*}
    is bounded from $L^{2}(\R^{d})$ to $H^{2}(\R^{d})$ with norm smaller
    than $\frac {C'}{d(z, \sigma(H))}$ for some $C' > 0$. Then, for
    all $\psi \in L^{2}(e^{\alpha \japx})$,
    \begin{align*}
      \|R(z) \psi\|_{H^{2}(e^{\alpha \japx})} = \|(z-H_{\alpha})^{-1} e^{\alpha \japx} \psi\|_{H^{2}(\R^{d})} \le \frac {C'}{d(z, \sigma(H))} \|\psi\|_{L^{2}(e^{\alpha \japx})}
    \end{align*}

  \end{proof}
\section*{Acknowledgments}
Discussions with Sören Behr, Luigi Genovese and Sonia Fliss are
gratefully acknowledged. This project has received funding from the
European Research Council (ERC) under the European Union’s Horizon
2020 research and innovation programme (grant agreement No 810367).

\printbibliography

\medskip

\noindent
$*$ Faculty of Mathematics, Technische Universität München, Germany
(dupuy@ma.tum.de)\\
$\dagger$ Inria Paris and Universit\'e Paris-Est, CERMICS, \'Ecole des Ponts ParisTech, Marne-la-Vall\'ee, France
(antoine.levitt@inria.fr)\\

\end{document}